\documentclass[10pt]{article}

\usepackage[latin1]{inputenc}
\usepackage{amsfonts}
\usepackage{amsmath}
\usepackage{amssymb}
\usepackage{amsthm}
\usepackage{mathrsfs}
\usepackage[american]{babel}
\usepackage{graphicx}
\newtheorem{theorem}{Theorem}[section]

\newtheorem{lemma}[theorem]{Lemma}

\author{\ \\ \\
Vikram Kamat\thanks{\texttt{vkamat@asu.edu}}\\
{\small School of Mathematical and Statistical Sciences}\\
{\small Arizona State University, Tempe, Arizona 85287-1804}\\ \\ \\
}

\title{Stability analysis for $k$-wise intersecting families}

\begin{document}

\maketitle

\newpage

\begin{abstract}
We consider the following generalization of the seminal Erd\H{o}s-Ko-Rado theorem, due to Frankl \cite{fr}. For some $k\geq 2$, let $\mathcal{F}$ be a $k$-wise intersecting family of $r$-subsets of an $n$ element set $X$, i.e. for any $F_1,\ldots,F_k\in \mathcal{F}$, $\cap_{i=1}^k F_i\neq \emptyset$. If $r\leq \dfrac{(k-1)n}{k}$, then $|\mathcal{F}|\leq {n-1 \choose r-1}$. We prove a stability version of this theorem, analogous to similar results of Dinur-Friedgut, Keevash-Mubayi and others for the Erd\H{o}s-Ko-Rado theorem. The technique we use is a generalization of Katona's circle method, initially employed by Keevash, which uses expansion properties of a particular Cayley graph of the symmetric group.

\end{abstract}

\pagebreak

\section{Introduction}
For a positive integer $n$, let $[n]=\{1,2,\ldots,n\}$. For positive integers $i$ and $j$ with $i\leq j$, let $[i,j]=\{i,i+1,\ldots,j\}$ ($[i,j]=\emptyset$ if $i>j$). Similarly let $(i,j]=\{i+1,\ldots,j\}$, which is empty if $i+1>j$. The notations $(i,j)$ and $[i,j)$ are similarly defined. Let ${[n] \choose r}$ be the family of all $r$-subsets of $[n]$. For $\mathcal{F}\subseteq {[n] \choose r}$ and $v\in [n]$, let $\mathcal{F}(v)=\{F\in \mathcal{F}:v\in F\}$, called a \textit{star} in $\mathcal{F}$, centered at $v$. A family $\mathcal{F}\subseteq {[n] \choose r}$ is called \textit{intersecting} if for any $A,B\in \mathcal{F}$, $A\cap B\neq \emptyset$. Similarly, call $\mathcal{F}\subseteq {[n] \choose r}$ $k$-wise intersecting if for any $F_1,\ldots, F_k\in \mathcal{F}$, $\bigcap_{i=1}^k F_i\neq \emptyset.$ Frankl \cite{fr} proved the following theorem for $k$-wise intersecting families.
\begin{theorem}[Frankl]\label{frankwise}
Let $\mathcal{F}\subseteq {[n] \choose r}$ be $k$-wise intersecting. If $r\leq \dfrac{(k-1)n}{k}$, then $|\mathcal{F}|\leq {n-1 \choose r-1}$.
\end{theorem}
It is trivial to note that the $k=2$ case of Theorem \ref{frankwise} is the seminal Erd\H{o}s-Ko-Rado theorem \cite{ekr}.
\begin{theorem}[Erd\H{o}s-Ko-Rado]\label{ekr}
Let $\mathcal{F}\subseteq {[n] \choose r}$ be intersecting. If $r\leq n/2$, then $|\mathcal{F}|\leq {n-1 \choose r-1}$.
\end{theorem}
\subsection{Stability}
The classical extremal problem is to determine the maximum size and structure of a family on a given ground set of size $n$ which avoids a given forbidden configuration $\mathcal{F}$. For example, the Erd\H{o}s-Ko-Rado theorem finds the maximum size of a set system on the set $[n]$, which does not have a pair of disjoint subsets. Often only a few trivial structures attain this extremal number. In case of the EKR theorem, the only extremal structure when $r<\frac{n}{2}$ is that of a star in ${[n] \choose r}$. A natural further step is to ask whether non-extremal families which have size close to the extremal number also have structure similar to any of the extremal structures. This approach was first pioneered by Simonovits \cite{sim} to answer a question in extremal graph theory and a similar notion for set systems was recently formulated by Mubayi \cite{mub}. Apart from being an interesting question in it's own right, this approach has found many applications, especially in extremal hypergraph theory, where exact results are typically much harder to prove.

One of the first stability results in extremal set theory was the theorem of Hilton and Milner \cite{hm} which proved a stability result for the Erd\"os-Ko-Rado theorem by giving an upper bound on the maximum size of non-star intersecting families. Other stability results for the Erd\H{o}s-Ko-Rado theorem have been recently proved by Dinur-Friedgut \cite{dinfried}, Keevash \cite{keev}, Keevash-Mubayi \cite{keevmub} and others. We prove the following stability result for Theorem \ref{frankwise}.
\begin{theorem}\label{thm1}
For some $k\geq 2$, let $1\leq r< \frac{(k-1)n}{k}$, and let $\mathcal{F}\subseteq {[n] \choose r}$ be a $k$-wise intersecting family. Then for any $0\leq \epsilon<1$, there exists a $0\leq \delta<1$ such that if $|\mathcal{F}|\geq (1-\delta){n-1 \choose r-1}$, then there is an element $v\in [n]$ such that $|\mathcal{F}(v)|\geq (1-\epsilon){n-1 \choose r-1}$.
\end{theorem}

We note that for $k\geq 2$, $\mathcal{F}$ is $k$-wise intersecting implies that it is intersecting. Hence if $r<n/2$, the results obtained in the papers mentioned above suffice as stability results for Theorem \ref{frankwise}. Consequently, the main interest of our theorem is in the structural information that it provides when $n/2\leq r<(k-1)n/k$. The technique we use to prove the theorem is a generalization of Katona's elegant proof of the EKR theorem \cite{katona}, initially employed by Keevash \cite{keev}, which uses expansion properties of a particular Cayley graph of the symmetric group.
\section{Proof of Theorem \ref{thm1}}\label{sec1}
Suppose $\mathcal{F}\subseteq {[n] \choose r}$ is a $k$-wise intersecting family, with $r<\frac{(k-1)n}{k}$. For any $0\leq \epsilon<1$, let $\delta=\frac{\epsilon}{2rn(n^3+1)}$ and suppose $|\mathcal{F}|\geq (1-\delta){n-1 \choose r-1}$. We will show that $\mathcal{F}$ contains a large star.
\subsection{Some Lemmas}
In this section, we will prove some Katona-type lemmas which we will employ later in the proof of the main theorem. We introduce some notation first. Consider a permutation $\sigma\in S_n$ as a sequence $(\sigma(1),\ldots,\sigma(n))$. We say that two permutations $\mu$ and $\pi$ are \textit{equivalent} if there is some $i\in [n]$ such that $\pi(x)=\mu(x+i)$ for all $x\in [n]$.\footnote{Addition is carried out mod $n$, so $x+i$ is either $x+i$ or $x+i-n$, depending on which lies in $[n]$.} Let $P_n$ be the set of equivalence classes, called \textit{cyclic} orders on $[n]$. For a cyclic order $\sigma$ and some $x\in [n]$, call the set $\{\sigma(x),\ldots,\sigma(x+r-1)\}$ a $\sigma$-interval of length $r$ \textit{starting} at $x$, \textit{ending} in $x+r-1$, and \textit{containing} the points ($x, x+1, \ldots, x+r-1$) (addition again mod $n$). The following lemma is due to Frankl \cite{fr}. We include the short proof below as we will build on these ideas in the proofs of the other lemmas.
\begin{lemma}[Frankl]\label{lem1}
Let $\sigma\in P_{n}$ be a cyclic order on $[n]$, and $\mathcal{F}$ be a $k$-wise intersecting family of $\sigma$-intervals of length $r\leq (k-1)n/k$. Then, $|\mathcal{F}|\leq r$.
\end{lemma}
 \begin{proof}
 Let $\mathcal{F}^c=\{[n]\setminus F:F\in \mathcal{F}\}$. Let $|\mathcal{F}|=|\mathcal{F}^c|=m$. We will prove that $m\leq r$. Since $r\leq (k-1)n/k$, we have $n\leq k(n-r)$. Suppose $G_1, \ldots, G_k\in \mathcal{F}^c$. Clearly $\cup_{i=1}^k G_i\neq [n]$; otherwise $\cap_{i=1}^k ([n]\setminus G_i)=\emptyset$, which is a contradiction. Let $G\in \mathcal{F}^c$. Without loss of generality, suppose $G$ ends in $n$. We now assign indices from $[1,k(n-r)]$ to sets in $\mathcal{F}^c$. For every set $G'\in \mathcal{F}^c\setminus \{G\}$, assign the index $x$ to $G'$ if $G'$ ends in $x$. Assign all indices in $[n,k(n-r)]$ for $G$. Consider the set of indices $[k(n-r)]$ and partition them into equivalence classes mod $n-r$. Suppose there is an equivalence class such that all $k$ indices in that class are assigned. Let $\{H_i\}_{i\in [k]}$ be the $k$ sets in $\mathcal{F}^c$ which end at the $k$ indices in the equivalence class. It is easy to note that $\cup_{i=1}^k H_i=[n]$, which is a contradiction. So for every equivalence class, there exists an index which has not been assigned to any set in $\mathcal{F}^c$. This implies that there are at least $n-r$ indices in $[k(n-r)]$ which are unassigned. Each set in $\mathcal{F}^c\setminus \{G\}$ has one index assigned to it, and $G$ has $k(n-r)-n+1$ indices assigned to it. This gives us $m-1+k(n-r)-n+1+n-r\leq k(n-r)$, which simplifies to $m\leq r$, completing the proof.

\renewcommand{\qedsymbol}{$\diamond$}
\end{proof}
We will now characterize the case when $|\mathcal{F}|=r$, in the following lemma.
\begin{lemma}\label{lem2}
Let $\sigma\in P_{n}$ be a cyclic order on $[n]$, and let $\mathcal{F}$ be a $k$-wise intersecting family of $\sigma$-intervals of length $r< (k-1)n/k$. If $|\mathcal{F}|= r$, then $\mathcal{F}$ consists of all intervals which contain a point $x$.
\end{lemma}
\begin{proof}
Without loss of generality, let $\sigma$ be the identity permutation and let $\mathcal{F}$ be a $k$-wise intersecting family of $\sigma$-intervals (henceforth, we drop the $\sigma$). As in the proof of Lemma \ref{lem1}, we consider $\mathcal{F}^c$ and assume (without loss of generality) that $F=\{r+1,r+2,\ldots,n\}\in \mathcal{F}^c$. It is clear from the proof of Lemma \ref{lem1} that if $|\mathcal{F}|=|\mathcal{F}^c|=r$, then there are exactly $n-r$ indices in $[k(n-r)]$, one from each equivalence class (modulo $n-r$), which are not assigned to any set in $\mathcal{F}^c$. In other words, no interval in $\mathcal{F}^c$ ends in any of these $n-r$ indices. Since $F$ ends in $n$, all indices in $[n,k(n-r)]$ (and there will be at least $2$, since $r<(k-1)n/k$) will be assigned. It will be sufficient to show that the set of unassigned indices is an interval $[x,x+n-r-1]$ for some $x\in [r]$. This would mean that no interval in $\mathcal{F}^c$ ends in any of the indices from $[x,x+n-r-1]$ and also that for every index $i\in [1,x-1]\cup [x+n-r,n]$, the interval ending in $i$ is a member of $\mathcal{F}^c$. This would imply that for every $i\in [n]$, there is an interval in $\mathcal{F}$ that \textit{begins} in index $i$ if and only if $i\in [1,x]\cup [x+n-r+1,n]$. This would mean that every interval in $\mathcal{F}$ contains $x$, as required.

Let $x$ be the smallest unassigned index in $[n-1]$. We will show that $[x,x+n-r-1]$ is the required interval containing all $n-r$ unassigned indices. Clearly $x\leq r$. Let $x\equiv j\textrm{ mod } n-r$. We will show that $x+i$ is unassigned for each $0\leq i\leq n-r-1$. We argue by induction on $i$, with the base case being $i=0$. Let $y=x+i$ for some $1\leq i\leq n-r-1$. Suppose $y$ is assigned, i.e. suppose there is a set $Y$ in $\mathcal{F}^c$ that ends in the index $y$. By the induction hypothesis, $y-1$ is unassigned. Let $E_{y-1}$ be the equivalence class containing $y-1$; since $n<k(n-r)$, we have $|E_{y-1}|\leq k$. As mentioned earlier, since $|\mathcal{F}^c|=r$, there are $n-r$ unassigned indices, exactly one from each equivalence class modulo $n-r$. In conjunction with the induction hypothesis, this means that every index in $E_{y-1}\setminus \{y-1\}$ is assigned to some interval in $\mathcal{F}^c$.

Let $I_1=E_{y-1}\cap (y-1,n]$. By the previous observation, each index in $I_1$ is assigned. Similarly, let $I_2=E_{y-1}\cap [1,y-1)$. Let $I_2'=\{j+1:j\in I_2\}$. $I_2'$ contains indices in the same equivalence class as $y$, and are assigned. This is true because all indices in $I_2'$ are smaller than $x$ and $x$ is the smallest unassigned index.\footnote{This is not true when $i>n-r-1$ and thus makes the induction ``stop'' at $i=n-r-1$.} Clearly, $E_{y-1}=I_1\cup I_2\cup \{y-1\}$ and consequently, $|E_{y-1}|=|I_1|+|I_2|+1$, giving $|I_1|+|I_2'|=|I_1|+|I_2|=|E_{y-1}|-1\leq k-1$. Let $J=I_1\cup I_2'$, so $|J|\leq k-1$ and all indices in $J$ are assigned. So let $\mathcal{H}$ be the subfamily of intervals in $\mathcal{F}^c$ which end in indices from $J$; we have $|\mathcal{H}|\leq k-1$ and hence the family $\mathcal{G}=\mathcal{H}\cup \{Y\}$ has at most $k$ sets. We will show that $\bigcup_{G\in \mathcal{G}}G=[n]$.

Let $p$ be the largest index in $I_1$ and let $q$ be the smallest index in $I_2'$. Now $q$ lies in the same equivalence class as $y$ and $p$ lies in the same equivalence class as $y-1$. If $n=k(n-r)$, it is easy to see that the set which ends in $q$ begins in the largest index from the same equivalence class as $y+1$, in other words, $p+2$. However, we have $n<k(n-r)$, so the set which ends in $q$ must contain $p+1$. This proves that the union of all sets in $\mathcal{G}$ is $[n]$, which is a contradiction. Thus $y$ is unassigned.

\renewcommand{\qedsymbol}{$\diamond$}
\end{proof}

Now let $\mathcal{F}\subseteq {[n] \choose r}$ be a $k$-wise intersecting family for some $r<\dfrac{(k-1)n}{k}$. For each cyclic order $\sigma \in P_n$, let $\mathcal{F}_{\sigma}$ be the subfamily of sets in $\mathcal{F}$ that are intervals in $\sigma$. We say that $\sigma$ is \textit{saturated} if $|\mathcal{F}_{\sigma}|=r$; otherwise call it unsaturated. By Lemma \ref{lem2}, if $\sigma$ is saturated, all sets in $\mathcal{F}_{\sigma}$ contain a common point, say $v$, so call $\sigma$ $v$-saturated to identify the common point.

For $i\leq n$, define an \textit{adjacent transposition} $A_i$ on a cyclic order $\sigma$ as an operation that swaps the elements in positions $i$ and $i+1$ ($i+1=1$ if $i=n$) of $\sigma$. We are now ready to prove our next lemma.
\begin{lemma}\label{lem3}
For $k\geq 2$, let $\mathcal{F}\subseteq {[n] \choose r}$ be a $k$-wise intersecting family with $r<\dfrac{(k-1)n}{k}$ and let $\sigma\in P_n$ be a $v$-saturated cyclic order. Let $\mu$ be the cyclic order obtained from $\sigma$ by an adjacent transposition $A_i$, $i\in [n]\setminus \{v,v-1\}$ ($v-1=n$ if $v=1$). If $\mu$ is saturated, then it is $v$-saturated.
\end{lemma}
\begin{proof}
As in the proof of Lemma \ref{lem2}, we let $\sigma$ be the \textit{identity} cyclic order $(1,2,\ldots,n)$ and suppose it is $n$-saturated, so $1\leq i\leq n-2$. Let $\mu=(1,\ldots,i-1,i+1,i,\ldots,n)$ be obtained from $\sigma$ by the adjacent transposition $A_i$ and let $\mu$ be saturated. As before, we consider the family of complements $\mathcal{F}^c$ and consider sets in this family which are intervals in the two cyclic orders. By Lemma \ref{lem2}, we know that for a $v$-saturated cyclic order, the set of the $n-r$ unassigned indices is $\{v,\ldots,v+n-r-1\}$. For $\sigma$, this interval is $\{n,1,\ldots,n-r-1\}$ as it is $n$-saturated. We will show that the interval of unassigned indices remains the same for $\mu$, thus proving that $\mu$ is also $n$-saturated.

Observe that there are only $2$ (out of $n$) intervals of length $n-r$ where $\sigma$ and $\mu$ differ in. First, the intervals which end in index $i$, i.e. $\{i-(n-r)+1,\ldots,i\}$ for $\sigma$ and $\{i-(n-r)+1,\ldots,i-1,i+1\}$ for $\mu$ and second, the intervals which begin in index $i+1$, i.e. $\{i+1,\ldots,i+n-r\}$ for $\sigma$ and $\{i,i+2,\ldots,i+n-r\}$ for $\mu$. In other words, only two indices, $i$ and $i+n-r$ can potentially change from assigned to unassigned, or vice-versa after the transposition $A_i$. We now consider three cases, depending on the value of $i$.
\begin{itemize}
\item Suppose $i\in (n-r-1,n-1)$. Since $i>n-r-1$, we assume that the index $i+n-r$ is assigned in $\mu$ and lies in the set $\{n,1,\ldots,n-r-2\}$ (since $i<n-1$, $i+n-r\neq n-r-1$). Suppose first that $i+n-r\neq n$. In this case, all indices in the set $A=\{n\}\cup [1,i+n-r)\cup (i+n-r,n-r-1]$ are unassigned in $\mu$. This is a contradiction, since $\mu$ is saturated and by Lemma \ref{lem2}, all unassigned indices \textit{must} occur in an interval of length $n-r$.
   \\
    So let $i+n-r=n$ be assigned in $\mu$, i.e. $\{i,i+2,\ldots,n\}\in \mathcal{F}^c$. Since $n$ is assigned and all indices in the interval $[1,n-r-1]$ are unassigned in $\mu$, by Lemma \ref{lem2}, the index $n-r$ must be unassigned in $\mu$; so $\{1,\ldots,n-r-1,\mu(n-r)\}\notin \mathcal{F}^c$. This is only possible if $i=n-r$ and consequently, $\mu=\{1,\ldots,n-r-1,n-r+1,n-r,\ldots,n\}$. Since $i+n-r=n$, this gives $n=2(n-r)$. Now, as $r<(k-1)n/k$, we must have $k\geq 3$. Now consider the following three intervals, each of length $n-r$: $\{1,\ldots,n-r\}$, $\{n-r,n-r+1,\ldots,n-1\}$ and $\{n-r,n-r+2,\ldots,n\}$. Note that the first two are intervals in $\sigma$ and since they both end in \textit{assigned} indices ($n-r$ and $n-1$ respectively) for $\sigma$, they are sets in $\mathcal{F}^c$. Similarly, the third set is an interval in $\mu$, ends in an assigned index $n$, and hence is a set in $\mathcal{F}^c$. The union of these three sets is $[n]$, a contradiction, completing the proof of this case.
\item Suppose $i\in [1,n-r-1)$. It is clear that the index $n-r-1$ stays unassigned in $\mu$, as the interval which ends in $n-r-1$ is the same in both cyclic orders, except in the order of elements. Also, if the index $i$ is assigned in $\mu$, the set of unassigned indices for $\mu$ would be some superset of $[1,i)\cup (i,n-r-1)$ not containing $i$; in other words, not of the form $[x,\ldots,x+n-r-1]$ for any $x\in [n]$, thus contradicting Lemma \ref{lem2}.\footnote{The case in which this can still satisfy Lemma \ref{lem2} is the trivial $r=1$. But this would imply $i=n$, a contradiction.} So the only way in which the set of unassigned indices can change is if $i+n-r=n$ and $n$ is assigned in $\mu$. Now the union of the two intervals $\{1,\ldots,n-r\}$ and $\{i,i+2,\ldots,n-r,\ldots,n\}$, both of which are sets in $\mathcal{F}^c$ (because $n-r$ and $n$ are assigned indices in $\sigma$ and $\mu$ respectively) is $[n]$, a contradiction.
\item Suppose $i=n-r-1$. In this case, the index $n-r$ is still assigned in $\mu$ because the interval ending in $n-r$ is the same in both cyclic orders, except the order of the elements. Using Lemma \ref{lem2}, this means that the set of $n-r$ unassigned indices in $\mu$ can be either $\{n,1,\ldots,n-r-1\}$ or $\{n-1,n,\ldots,n-r-2\}$. If the set is the same as in $\sigma$, we are done, so suppose it is $\{n-1,n,\ldots,n-r-2\}$. This means that $n-r-1$ is assigned in $\mu$ and $n-1$ is unassigned in $\mu$. This is only possible if $i+n-r=n-1$. This means $n=2(n-r)$ and $k\geq 3$. Now consider the following three intervals: $\{1,\ldots,n-r\}$, $\{n-r,n-r+1,\ldots,n-1\}$ and $\{n,1,\ldots,n-r-2,n-r\}$. The first two sets are intervals in $\sigma$ and end in assigned indices ($n-r$ and $n-1$ respectively) for $\sigma$, while the third set is an interval in $\mu$ which ends in an assigned index $i=n-r-1$. Thus, all three sets lie in $\mathcal{F}^c$. The union of these three sets is clearly $[n]$, a contradiction.
\end{itemize}

\renewcommand{\qedsymbol}{$\diamond$}
\end{proof}
\subsection{Cayley Graphs}
In this small section, we gather some facts about expansion properties of a specific Cayley graph of the symmetric group. We will consider the Cayley graph $G$ on $S_{n-1}$ generated by the set of adjacent transpositions $A=\{(12),\ldots,(n-2\textrm{ }n-1)\}$. In particular, the vertex set of $G$ is $S_{n-1}$ and two permutations $\sigma$ and $\mu$ are adjacent if $\mu=\sigma\circ a$, for some $a\in A$. We note that the transposition operates by exchanging adjacent positions (as opposed to consecutive values). $G$ is an $n-2$-regular graph. It was shown by Keevash \cite{keev}, using a result of Bacher \cite{bach}, that $G$ is an $\alpha$-expander for some $\alpha>\frac{1}{n^3}$, i.e. for any $H\subseteq V(G)$ with $|H|\leq \frac{|V(G)|}{2}$, we have $N(H)\geq \alpha|H|> \frac{|H|}{n^3}$, where $N(H)$ is the set of all vertices in $V(G)\setminus H$ which are adjacent to some vertex in $H$.
\subsection{Proof of Main Theorem}
\begin{proof}[Proof of Theorem \ref{thm1}]
We will finish the proof of Theorem \ref{thm1} in this section. We can identify every cyclic order in $P_n$ with a permutation $\sigma\in S_n$ having $\sigma(n)=n$. Restricting $\sigma$ to $[n-1]$ gives a bijection between $P_n$ and $S_{n-1}$. Let $U$ be the set of unsaturated cyclic orders in $P_n$. We have
\begin{eqnarray*}
r!(n-r)!|\mathcal{F}| &=& \sum_{\sigma\in P_n}|\mathcal{F}_{\sigma}| \\
&\leq & \sum_{\sigma\in P_n}r-|U| \\
&= & r(n-1)!-|U|.
\end{eqnarray*}
This gives us $|U|\leq r(n-1)!-r!(n-r)!(1-\delta){n-1 \choose r-1}=r\delta(n-1)!$, implying that there are at least $(1-r\delta)(n-1)!$ saturated orders in $P_n$.

We now consider the Cayley graph $G$ defined above, with the vertex set being $P_n$ and the generating set being the set of adjacent transpositions $A=\{(12),\ldots,(n-2\textrm{ }n-1)\}$. Suppose $S$ is a subset of saturated cyclic orders. We can use the expansion property of $G$ to conclude that if $n^3r\delta\leq \frac{|S|}{(n-1)!}\leq \frac{1}{2}$, we get $N(S)>|S|/n^3\geq r\delta(n-1)!$. This means that there is a saturated cyclic order in $N(S)$. We will use this observation to show that the subgraph of $G$ induced by the set of all saturated cyclic orders, say $H$, has a large component. Consider the set of all components in $H$. Now a component in $H$ can be either \textit{small}, i.e. have size at most $n^3r\delta(n-1)!$ or be \textit{large}, i.e. have size bigger than $(n-1)!/2$. Clearly there can be at most one large component. We argue that the total size of all small components is at most $n^3r\delta(n-1)!$. Suppose not. Let $S'$ be the union of (at least $2$) small components such that $n^3r\delta(n-1)!\leq |S'|\leq 2n^3r\delta(n-1)!\leq (n-1)!/2$. Now using the above observation, $N_H(S')$ is non-empty, a contradiction. Thus there is a large component of size at least $(1-n^3r\delta)(n-1)!$. Call this component $H'$. Suppose $\sigma$ is a $v$-saturated cyclic order in $H'$. By Lemma \ref{lem3}, every cyclic order in $H'$ is $v$-saturated. Thus, $r!(n-r)!|\mathcal{F}(v)|\geq \sum_{\sigma\in H'}|\mathcal{F}_{\sigma}|\geq r(1-r\delta-n^3r\delta)(n-1)!$, which gives $|\mathcal{F}(v)|\geq (1-\frac{\epsilon}{2n}){n-1 \choose r-1}$, since $\delta = \dfrac{\epsilon}{2rn(n^3+1)}$.
\end{proof}

\textbf{Remark:} The proof of Theorem \ref{thm1} also contains a proof of the structural uniqueness of the extremal configurations for Theorem \ref{frankwise} when $r<(k-1)n/k$. This can be easily observed by putting $\epsilon=0$ in the statement of the theorem, or by just using Lemmas \ref{lem1}, \ref{lem2} and \ref{lem3}. We note that the original proof by Frankl in \cite{fr} did not include this structural information. However in \cite{shift}, Frankl gives another proof of Theorem \ref{frankwise} using the Kruskal-Katona theorem, which includes the characterization of the extremal structures for $r\leq (k-1)n/k$ when $k\geq 3$ and $r<(k-1)n/k$ when $k=2$. An alternate proof of this characterization is also given by Mubayi and Verstraete \cite{mubver}.
\section*{Acknowledgement}
The author wishes to thank Glenn Hurlbert for help in the writing of this paper, Andrzej Czygrinow for the many productive discussions on stability analysis and the anonymous referee for several helpful suggestions.


\begin{thebibliography}{99}

\bibitem{bach} R. Bacher, Valeur propre minimale du laplacien de Coxeter pour le groupe sym\'etrique, J. Algebra 167 (1994), 460-472.

\bibitem{df} M. Deza, P. Frankl, Erd\H{o}s-Ko-Rado theorem -- 22
years later, SIAM J. Algebraic Discrete Methods 4 (1983), no. 4,
419-431.

\bibitem{dinfried} I. Dinur, E. Friedgut, Intersecting families are essentially contained in juntas, submitted for publication.

\bibitem{ekr} P. Erd\H{o}s, C. Ko, R. Rado, Intersection theorems
for systems of finite sets, Quart. J. Math Oxford Ser. (2) 12(1961),
313-320.

\bibitem{fr} P. Frankl, On Sperner systems satisfying an additional condition, J. Combin. Th. A 20 (1976), 1-11.

\bibitem{shift} P. Frankl, The shifting technique in extremal set
theory,  Surveys in combinatorics 1987 (New Cross, 1987), London
Math. Soc. Lecture Note Ser., 123, Cambridge Univ. Press, Cambridge,
1987, 81-110.

\bibitem{hm} A. J. W. Hilton, E. C. Milner, Some intersection
theorems for systems of finite sets, Quart. J. Math. Oxford 18
(1967), 369-384.

\bibitem{katona} G. O. H. Katona, A simple proof of the
Erd\H{o}s--Ko--Rado theorem, J. Combin. Theory Ser. B 12 (1972)
183-184.

\bibitem{keev} P. Keevash, Shadows and intersections: Stability and new proofs, Advances in Mathematics 218 (2008), 1685-1703.

\bibitem{keevmub} P. Keevash, D. Mubayi, Set systems without a simplex or a cluster, Combinatorica 30 (2010), 175-200

\bibitem{mub} D. Mubayi, Structure and stability of triangle-free set systems, Transactions of the American Mathematical Society 359 (2007), 275-291.

\bibitem{mubver} D. Mubayi, J. Verstraete, Proof of a conjecture of Erd\H{o}s on triangles in set systems, Combinatorica, 25 (2005), no. 5, 599-614.

\bibitem{sim} M. Simonovits, A method for solving extremal problems in graph theory, stability problems. 1968 Theory of Graphs (Proc. Colloq., Tihany, 1966) pp. 279-319, Academic Press, New York.

\end{thebibliography}
\end{document}